\documentclass[12pt]{amsart}
\usepackage{amsfonts}
\usepackage[cp1250]{inputenc}
\usepackage{amssymb}
\usepackage{amsmath}
\usepackage{amsthm}
\theoremstyle{plain}
\newtheorem{theorem}{Theorem}
\newtheorem*{prop}{Proposition}

\newtheorem{lemma}{Lemat}
\newtheorem*{wn}{Corollary}

\newtheorem{defin}[theorem]{Definition}

\date{\today}

\begin{document}
\title{On balayage and B-balayage operators}
\author{Maria Nowak and Pawe\l \ Sobolewski}
 \maketitle
\begin{abstract}
Here we consider the balayage operator in the setting of $H^p$
spaces and its Bergman space version (B-balayage) introduced by H.
Wulan, J. Yang and K. Zhu \cite{WYZ}, and extend some known results
on these operators.
\end{abstract}

\section{Introduction}

Let $\mathbb D$ denote the unit disk $\{z:\mathbb C:|z|<1\}$ and
$\mathbb T$ the unit circle. For $0<p<\infty$, the Hardy space $H^p$
consists of all functions $f$ which are holomorphic on $\mathbb D$
and satisfy
$$\|f\|_{H^p}=\sup_{0<r<1}\left\{\frac{1}{2\pi}\int_0^{2\pi}|f(re^{it})|^p dt\right\}^\frac 1p<\infty.$$
 It is known that each function $f\in H^p$ has the
radial limit $f(e^{it})=\lim_{r\to 1^-}f(re^{it})$ a.e. on $\mathbb
T$  and $f(e^{it})\in L^p(\mathbb T)$.

For $\phi\in L^1(\mathbb T)$, we say that $\phi\in BMO(\mathbb T)$
if
$$\|\phi\|_*=\sup_{I\subset \mathbb T}\frac{1}{|I|}\int_I
|\phi(e^{it})-\phi_I|dt<\infty,$$ where $I$ denotes any arc of
$\mathbb T$, $|I|$ is its arc length  and
$$\phi_I=\frac{1}{|I|}\int_I\phi(e^{it})dt.$$

 In \cite{Xiao Yuan} the authors  have recently considered
Campanato spaces  $\mathcal{L}^{p,\lambda}(\mathbb T)$ defined as
follows. For $\lambda \geq 0$ and $1\leq p<\infty$, the space
$\mathcal{L}^{p,\lambda}(\mathbb T)$ consists of all functions
$\phi\in L^p(\mathbb T)$ for which
$$\sup_{I\subset \mathbb T}\frac{1}{|I|^\lambda}\int_I
|\phi(e^{it})-\phi_I|^pdt<\infty.$$ We note that $BMO(\mathbb
T)=\mathcal{L}^{p,1},\ 1 \leq p<\infty,$ (see, e.g., \cite[pp.
222-235]{Garnet}).

 For a finite positive Borel measure $\mu$ on $\mathbb D$,
the function
\begin{equation}\label{bal}
S_\mu (e^{it})=\int_\mathbb D
\frac{1-|z|^2}{|1-ze^{-it}|^2}d\mu(z),\end{equation} \emph{is called
the balayage of} $\mu$. It follows from Fubini's theorem that $S_\mu
(e^{it})\in L^1(\mathbb T)$ (see, \cite[p. 229]{Garnet}).

 If $I\subset$ is an arc of $\mathbb T$, the Carleson square $S(I)$
is defined as
$$S(I)=\{re^{it}:e^{it}\in I, 1-\frac{|I|}{2\pi}\leq r<1\}.$$

 A positive Borel measure
$\mu$ is called an $s$-Carleson measure, $0<s<\infty$, if there
exists a positive constant $C=C(\mu)$ such that
$$\mu(S(I))\leq C(\mu)|I|^s,\quad \text{for any arc } I\subset \mathbb
T.$$ A 1-Carleson measure is simply called a Carleson measure. In
\cite{Carleson} Carleson  proved that if $\mu$ is a positive Borel
measure in $\mathbb D$, then for $0<p<\infty$, $H^p\subset
L^p(d\mu)$ if and only if $\mu$ is a Carleson measure.

It has been proved in \cite[p. 229]{Garnet} that if $\mu$ is the
Carleson measure, then $S_\mu$ belongs to $BMO(\mathbb T)$. However,
the Carleson property of measure $\mu$ is not a necessary condition
for $S_\mu$  being a $BMO(\mathbb T)$ function (\cite{Pott}).

In the next section we obtain an extension of the above mentioned
result. More exactly, we prove that if $\mu$ is an s-Carleson
measure, $0\leq s<1$, then $S_\mu$ belongs to $\mathcal{L}^{1,s}$.

In \cite{WYZ} H. Wulan, J. Yang and K. Zhu  introduced the Bergman
space
 version of the balayage operator on the unit disk that was called
 B-balayage. The B-balayage of a finite complex measure $\mu$ on $\mathbb D$
 is given by $$G_\mu(z)=\int_\mathbb D \frac{(1-|w|^2)^2}{|1-\bar zw|^4} d\mu(w), \quad z\in\mathbb D.$$
 It has been proved in  \cite{WYZ} that if $\mu$ is a
2-Carleson measure, then there exists a constant $C>0$ such that
$$|G_{\mu}(z)-G_{\mu}(w)|\leq C \beta(z,w),\quad z,w\in\mathbb D,$$
where $\beta$ is the hyperbolic metric on $\mathbb D$. Here,
applying a similar idea to that used in the proof of this result, we
prove
\begin{theorem}\label{b-balayage}
Assume that $1< p<\infty$ and  $\mu$ is a positive Borel measure on
$\mathbb D$. If $\mu$ is a $2p$-Carleson measure, then there exists
a positive constant $C=C(p)$ such that
$$|G_{\mu}(z)-G_{\mu}(w)|\leq C\,\left(\beta(z,w)\right)^{\frac 1p}$$
for all $z,w\in\mathbb D$.
\end{theorem}
Here $C$ will denote a positive constant which can vary from line to
line.

\section{Balayage operators and Campanato spaces $\mathcal{L}^{1,s}$}

We start with the following
\begin{theorem}\label{tw 1}
If $\mu$ is an $s$-Carleson measure, $0<s\leq 1$,  $S_\mu$ is given
by (\ref{bal}) and $0\leq \gamma<1$, then there exists a positive
constant $C$ such that for any $I\subset \mathbb T$
$$
\frac{1}{|I|^{1+s-\gamma}}\int_I\int_I\frac{|S_\mu(e^{i\theta})-S_\mu(e^{i\varphi})|}{|e^{i\theta}-e^{i\varphi}|^\gamma}\,d\theta\,d\varphi\leq
C.
$$
\end{theorem}

\begin{proof}
Without loss of generality we can assume that $|I|<1$.

 Let
for $z\in\mathbb D$ and $\theta\in\mathbb R$
$$P_z(\theta)=\frac{1-|z|^2}{|1-ze^{-i\theta}|^2}=\text{Re}\left(\frac{1+ze^{-i\theta}}{1-ze^{-i\theta}}\right)$$
be the Poisson kernel for  the disk $\mathbb D$. By the Fubini
theorem,
\begin{eqnarray}\label{wzor 3}
&&\int_I\int_I\frac{|S_\mu(e^{i\theta})-S_\mu(e^{i\varphi})|}{|e^{i\theta}-e^{i\varphi}|^\gamma}\,d\theta\,d\varphi\leq
\int_I\int_I\int_{\mathbb
D}\frac{|P_z(\theta)-P_z(\varphi)|}{|e^{i\theta}-e^{i\varphi}|^\gamma}d\mu(z)\,d\theta\,d\varphi\ \nonumber\\
&&=\int_{\mathbb
D}\int_I\int_I\frac{|P_z(\theta)-P_z(\varphi)|}{|e^{i\theta}-e^{i\varphi}|^\gamma}\,d\theta\,d\varphi\
d\mu(z)
\end{eqnarray}

For a subarc $I$ of $ \mathbb T$ let $2^nI,\ n\in\mathbb N$, denote
the subarc of $\mathbb T$ with the same center as $I$ and the length
$2^n|I|$.

In view of the equality
$$\int_0^{2\pi}P_z(\theta)d\theta=2\pi,$$
we have
\begin{eqnarray*}
\int_I\int_I
\frac{P_z(\theta)}{|e^{i\theta}-e^{i\varphi}|^\gamma}\,d\theta\,d\varphi&=&
\int_IP_z(\theta)\int_I
\frac{d\varphi}{|e^{i\theta}-e^{i\varphi}|^\gamma}\,d\theta \leq
C|I|^{1-\gamma}.
\end{eqnarray*}
Consequently, \begin{align}\label{wzor 4} \int_{S(2I)}\int_I\int_I
\frac{|P_z(\theta)-P_z(\varphi)|}{|e^{i\theta}-e^{i\varphi}|^\gamma}\,d\theta\,d\varphi\
d\mu(z)\leq 2C|I|^{1-\gamma}\int_{S(2I)}d\mu(z)\leq
C|I|^{1+s-\gamma}. \end{align}
 Since $P_z(\theta)\leq 4$ for
$|z|\leq\frac 12$, we get
\begin{eqnarray*}
&&\int_{|z|\leq \frac 12
}\int_I\int_I\frac{|P_z(\theta)-P_z(\varphi)|}{|e^{i\theta}-e^{i\varphi}|^\gamma}\,d\theta\,d\varphi\
d\mu(z) \leq  8\mu(\mathbb D)\int_I\int_I\frac{d\theta
d\varphi}{|e^{i\theta}-e^{i\varphi}|^\gamma} \\
&&\leq C|I|^{2-\gamma}\leq C |I|^{1+s-\gamma}.
\end{eqnarray*}

Now we assume that $\frac 12\leq |z|<1$ and $z=|z|e^{i\omega}\in
S(2^{n+1}I)\setminus S(2^nI)$. We consider two cases: $(i)$
$e^{i\omega} \in 2^nI$ and $(ii)$ $e^{i\omega} \in 2^{n+1}I\setminus
2^nI$.

 In case $(i)$ we have
$$\frac{2^n|I|}{2\pi}<1-|z|\leq \frac{2^{n+1}|I|}{2\pi}.$$
Thus
\begin{eqnarray*}
|P_z(\theta)-P_z(\varphi)|&=&\frac{(1-|z|^2)2|z||\cos(\theta-\omega)-\cos(\varphi-\omega)|}{\left((1-|z|)^2+4|z|\sin^2\frac{\theta-\omega}{2}\right)
\left((1-|z|)^2+4|z|\sin^2\frac{\varphi-\omega}{2}\right)} \\
&\leq&\frac{8|\sin\frac{(\theta-\omega)+(\varphi-\omega)}{2}||\sin\frac{(\theta-\varphi)}{2}|}{(1-|z|)^3}\\
&\leq&2\frac{\left(|\theta-\omega|+|\varphi-\omega|\right)|\theta-\varphi|}{(1-|z|)^3}.
\end{eqnarray*}
So, if  $e^{i\theta},e^{i\varphi}\in I$,  then
\begin{eqnarray}\label{wzor 1}
\frac{|P_z(\theta)-P_z(\varphi)|}{|e^{i\theta}-e^{i\varphi}|^\gamma}&\leq&C\frac{\left(|\theta-\omega|+|\varphi-\omega|\right)|\theta-\varphi|^{1-\gamma}}{(1-|z|)^3}\nonumber\\
&\leq&C\frac{2^n|I||I|^{1-\gamma}}{(2^n|I|)^3}=C\frac{|I|^{-1-\gamma}}{2^{2n}}.
\end{eqnarray}
Now we turn to case $(ii)$. Then for $e^{i\psi}\in I$,
$$2^{n-2}|I|\leq |\psi-\omega|\leq 2^n|I|.$$
Consequently, for $e^{i\theta},e^{i\varphi}\in I$, we get
\begin{eqnarray}
\label{wzor
2}\frac{|P_z(\theta)-P_z(\varphi)|}{|e^{i\theta}-e^{i\varphi}|^\gamma}&\leq&
2\frac{\left||(1-ze^{-i\theta}|^2-|1-ze^{-i\varphi}|^2\right|}{|e^{i\theta}-e^{i\varphi}|^\gamma|1-ze^{-i\theta}||1-ze^{-i\varphi}|^2}\nonumber \\
&\leq&C
\frac{\left(|\theta-\omega|+|\varphi-\omega|\right)|\theta-\varphi|^{1-\gamma}}{|\theta-\omega||\varphi-\omega|^2}\nonumber \\
&\leq& C\frac{|I|^{-1-\gamma}}{2^{2n}}.
\end{eqnarray}

Now we put $Q_n=S(2^nI),\ n=1,2,\ldots$ Then by (\ref{wzor 1}) and
(\ref{wzor 2}),

\begin{eqnarray*}
&&\int_{Q_{n+1}\setminus Q_n\atop |z|\geq \frac 12}\int _I\int_I
\frac{|P_z(\theta)-P_z(\varphi)|}{|e^{i\theta}-e^{i\varphi}|^\gamma}d\theta
d\varphi d\mu(z) \leq C
\frac{|I|^{1-\gamma}}{2^{2n}}\int_{\substack{Q_{n+1}}}d\mu(z)\leq C
\frac{|I|^{1+s-\gamma}}{2^{n(2-s)}}.
\end{eqnarray*}
The above inequality and (\ref{wzor 4}) imply \begin{eqnarray*}
&&\int_{\mathbb
D}\int_I\int_I\frac{|P_z(\theta)-P_z(\varphi)|}{|e^{i\theta}-e^{i\varphi}|^\gamma}\,d\theta\,d\varphi\
d\mu (z)\leq \int_{Q_{1}}\int_I\int_I\frac{|P_z(\theta)-P_z(\varphi)|}{|e^{i\theta}-e^{i\varphi}|^\gamma}d\theta d\varphi d\mu(z)\\
&&+ \sum_{n=1}^\infty \int_{Q_{n+1}\setminus
Q_n}\int_I\int_I\frac{|P_z(\theta)-P_z(\varphi)|}{|e^{i\theta}-e^{i\varphi}|^\gamma}d\theta d\varphi d\mu(z)\\
&&\leq C|I|^{s+1-\gamma}\sum_{n=1}^\infty \frac{1}{2^{n(2-s)}}
=C|I|^{1+s-\gamma}.
\end{eqnarray*}

\end{proof}

The next theorem shows that if $\mu$ is an s-Carleson measure,
$0<s\leq 1$, then $S_\mu$ is in the Campanato space $\mathcal
{L}^{1,s}$.
\begin{theorem}
If $\mu$ is an $s$-Carleson measure on $\mathbb D$,  $0<s\leq 1$ and
$S_\mu(t)=S_\mu(e^{it})$ is the balayage operator of $\mu$ given by
(\ref{bal}), then there exists a positive constant $C$ such that for
any $I\subset \mathbb T$
$$\frac{1}{|I|^s}\int_I|S_\mu(t)-(S_\mu)_I|dt\leq C.$$
\end{theorem}

\begin{proof}
It is enough to observe that
\begin{eqnarray*}
&&\frac{1}{|I|^s}\int_I|S_\mu(t)-(S_\mu)_I|dt \leq
\frac{1}{|I|^{s+1}}\int_I\int_I|S_\mu(t)-S_\mu(u)|dtdu
\end{eqnarray*}
and the inequality follows from Theorem \ref{tw 1} with $\gamma =
0$.
\end{proof}

\section{B-balayage for weighted Bergman spaces $A^p_\alpha$}

Recall that for $0<p<\infty$, $-1<\alpha<\infty$, the weighted
Bergman space $A_\alpha^p$ is the space of all holomorphic functions
in $L^p(\mathbb D,dA_\alpha)$, where
$$dA_\alpha(z)=(\alpha+1)(1-|z|^2)^\alpha dA(z)$$ and $dA$ is the
normalized Lebesgue measure on $\mathbb D$, that is $\int_\mathbb D
dA=1$. If $f$ is in $L^p(\mathbb D, dA_\alpha)$, we write
$$\|f\|_{p,\alpha}^p=\int_\mathbb D |f(z)|^p dA_\alpha(z)$$
It is well known that for $1<p<\infty$ the Bergman projection
$P_\alpha$ given by
$$P_\alpha f(z)=\int_\mathbb D \frac{f(w)}{(1-z\bar w)^{2+\alpha}}dA_\alpha (w)$$
is a bounded operator from $L^p(\mathbb D, dA_\alpha)$ onto
$A_\alpha ^p$.

 Let for $z,w\in\mathbb D$, the function
$$\varphi_z(w)=\frac{z-w}{1-\bar z w}$$ denote the
automorphism of the unit disk $\mathbb D$. The hyperbolic metric on
$\mathbb D$ is given by
$$\beta(z,w)=\frac 12
\log\frac{1+|\varphi_z(w)|}{1-|\varphi_z(w)|}.$$ For $z\in\mathbb D$
and $r>0$ the hyperbolic disk with center $z$ and radius $r$  is
$$D(z,r)=\{w\in\mathbb D:\beta(z,w)<r\}.$$
For $s>1$ the condition for an s-Carleson measure given in
Introduction is equivalent to the condition where Carleson squares
are replaced by hyperbolic disks. More exactly, we have the
following
\begin{prop}\cite{Zhu b,DS}
Let $\mu$ be a positive Borel measure on $\mathbb D$ and
$1<s<\infty$. Then the following statements are equivalent
\begin{itemize}
\item[(i)] $\mu$ is an s-Carleson measure
\item[(ii)] $\mu(D(z,r))\leq C(1-|z|^2)^s$ for some constant $C$ depending only on $r$ for all hyperbolic disk
$D(z,r)$, $z\in\mathbb D$.
\end{itemize}
\end{prop}
For $\alpha >-1$,  $(\alpha +2)$ measures are characterized by the
following result.
\begin{theorem} \cite[p. 40]{KHZ}
Suppose $\mu$ is a positive Borel measure on $\mathbb D$. Then $\mu$
is an $(\alpha+2)$-Carleson measure if and only if there exists the
positive constant $C$ such that
$$\int_\mathbb D |f(z)|^p d\mu(z)\leq C\int_\mathbb D |f(z)|^p
dA_\alpha(z)$$ for all $f\in A_\alpha^p$.
\end{theorem}

The next corollary is an immediate consequence of  Proposition.
\begin{wn}\cite{WYZ}
For  $\alpha>-1,\ \sigma>0$, let $\mu,\nu$ be positive Borel
measures on $\mathbb D$ such that  $$d\nu(z)=(1-|z|)^\sigma
d\mu(z).$$ Then $\mu$ is an $A^p_\alpha$-Carleson measure if and
only if $\nu$ is an $A^p_{\alpha+\sigma}$-Carleson measure.

\end{wn}

Recall that for $1<p<\infty$, the Besov space $B_p$ is the space of
all functions  $f$ analytic  on $\mathbb D$ and such that
$$\|f\|_{B_p}^p=\int_\mathbb D
|f'(z)|^p(1-|z|^2)^p d\tau(z)<\infty,$$ where
$$d\tau(z)=\frac{dA(z)}{(1-|z|^2)^2}$$ is the M\"obius invariant
measure on $\mathbb D$.

 We will use the fact that the Besov space
 $ B_p=P_\alpha (L^p,d\tau).$
The proof of this equality for $\alpha=0$ is given in \cite[p.
90-92]{Z1} and similar arguments can be used for $\alpha>-1$. In
particular, if $f=P_\alpha g$, where $g\in L^p(d\tau)$, then
$$(1-|z|^2)f'(z)=(\alpha+2)(1-|z|^2)\int_\mathbb D \frac{g(w)\bar w}{(1-z\bar
w)^{3+\alpha}}dA_\alpha (w).$$ It then follows from \cite[Theorem
1.9]{KHZ} that \begin{equation}\label{nier bp}\|f\|_{B_p}\leq
C_{p,\alpha}\|g\|_{L^p(d\tau)}.\end{equation}

 The next theorem gives a Lipschitz type estimate
for functions in the analytic Besov space.

\begin{theorem}\cite{Zhu ks}\label{tw 5}
For any $1<p<\infty$, there exists a constant $C_p>0$ such that
$$|f(z)-f(w)|\leq C_p\|f\|_{B_p}(\beta(z,w))^\frac 1q$$ for all $f\in
B_p$ and $z,w\in\mathbb D$, where $\frac 1p+\frac 1q =1$.
\end{theorem}

\begin{proof}[Proof of Theorem \ref{b-balayage}] For $z,w$ we have
\begin{eqnarray*}
|G_{\mu}(z)-G_{\mu}(w)|& \leq&  \int_{\mathbb D}\left|\frac{(1-|a|^2)^2}{|1-a\bar z|^4}-\frac{(1-|a|^2)^2}{|1-a\bar w|^4}\right|d\mu(a)\\
& \leq&  \int_{\mathbb D}\left|\frac{(1-|a|^2)^2}{(1-a\bar z)^4}-\frac{(1-|a|^2)^2}{(1-a\bar w)^4}\right|d\mu(a).
\end{eqnarray*}
Since $\mu$ is a finite measure on $\mathbb D$, the Jensen
inequality yields
\begin{eqnarray*}
|G_{\mu}(z)-G_{\mu}(w)|^p&\leq& C \int_{\mathbb D}\left|\frac{(1-|a|^2)^2}{(1-a\bar z)^4}-\frac{(1-|a|^2)^2}{(1-a\bar w)^4}\right|^pd\mu(a)\\
&=& C \int_{\mathbb D}\left|\frac{1}{(1-a\bar z)^4}-\frac{1}{(1-a\bar w)^4}\right|^p(1-|a|^2)^{2p}d\mu(a).\\
\end{eqnarray*}
Let $q>1$ be the conjugate index for $p$, that is, $\frac 1p +\frac
1q =1$. Then $2p=2+\frac{2}{q-1}$ and under the assumption, $\mu$ is
an $A^r_{\frac{2}{q-1}}$-Carleson measure, $0<r<\infty$. By
Corollary, $(1-|a|^2)^{2p}d\mu(a)$ is an
$A^r_{2p+\frac{2}{q-1}}$-Carleson measure. Consequently,
\begin{eqnarray*}
&&\int_{\mathbb D}\left|\frac{1}{(1-a\bar z)^4}-\frac{1}{(1-a\bar w)^4}\right|^p(1-|a|^2)^{2p}d\mu(a)\\
&&\leq C \int_{\mathbb D}\left|\frac{1}{(1-a\bar z)^4}-\frac{1}{(1-a\bar w)^4}\right|^p(1-|a|^2)^{2p+\frac{2}{q-1}}dA(a)\\
&&= C \int_{\mathbb D}\left|\frac{(1-|a|^2)^2}{(1-\bar
az)^4}-\frac{(1-|a|^2)^2}{(1-\bar a w)^4}\right|^pdA_{\frac
{2}{q-1}}(a).
\end{eqnarray*}
Now set $\alpha=\frac{2}{q-1}$ and note that
\begin{eqnarray*}
&&\int_{\mathbb D}\left|\frac{(1-|a|^2)^2}{(1-\bar az)^4}-\frac{(1-|a|^2)^2}{(1-\bar aw)^4}\right|^p dA_{\alpha}(a)\\
&&= \left(\sup_{\|f\|_{q,\alpha}\leq 1}\left|\int_{\mathbb
D}\left(\frac{(1-|a|^2)^2}{(1-\bar az)^4}-\frac{(1-|a|^2)^2}{(1-\bar
a w)^4}\right)f(a)dA_{\alpha}(a)\right|\right)^p.
\end{eqnarray*}
Put $g=(\alpha+1)^\frac 1q(1-|a|^2)^{\frac{\alpha+2}{q}}f$ and
observe that $\|f\|_{q,\alpha}\leq 1$ if and only if $
\|g\|_{L^q(d\tau)}\leq 1$. Moreover, since $\alpha=\frac{2}{q-1}$
satisfies $\frac{\alpha+2}{q}=\alpha$, we get
\begin{eqnarray*}
&&\sup_{\|f\|_{q,\alpha}\leq 1}\left|\int_{\mathbb D}\left(\frac{(1-|a|^2)^2}{(1-\bar az)^4}-\frac{(1-|a|^2)^2}{(1-\bar a w)^4}\right)f(a)dA_{\alpha}(a)\right|\\
&&= C\sup_{\|g\|_{L^q(d\tau)}\leq 1}\left|\int_{\mathbb D}\left(\frac{(1-|a|^2)^2}{(1-\bar az)^4}-\frac{(1-|a|^2)^2}{(1-\bar a w)^4}\right)g(a)dA(a)\right|\\
&&= C\sup_{\|g\|_{L^q(d\tau)}\leq 1}\left|\int_{\mathbb D}\left(\frac{g(a)}{(1-\bar az)^4}-\frac{g(a)}{(1-\bar a w)^4}\right)dA_2(a)\right|\\
&&= C\sup_{\|g\|_{L^q(d\tau)}\leq 1}\left|P_2g(z)-P_2g(w)\right|\leq
C(\beta(z,w))^\frac 1p,
\end{eqnarray*}
where the last inequality follows from Theorem \ref{tw 5} and
inequality (\ref{nier bp}).
\end{proof}

\vskip20pt

Maria Nowak, Pawe\l \ Sobolewski

Instytut Matematyki  UMCS

pl. Marii Curie-Sk\l odowskiej 1

20-031 Lublin, Poland

\email{mt.nowak@poczta.umcs.lublin.pl}

\email{pawel.sobolewski@umcs.eu}

\end{document}